\documentclass[reqno]{amsart}

\usepackage{amssymb}
\usepackage{amsmath}

\usepackage[shortlabels]{enumitem}

\usepackage{comment}

\usepackage{mathtools}
\usepackage{xparse}
\usepackage{xifthen}

\usepackage{hyperref}
\usepackage{amsthm, thmtools}

\usepackage[colorinlistoftodos, textsize=small]{todonotes}
\makeatletter
\renewcommand{\todo}[2][]{\tikzexternaldisable\@todo[#1]{#2}\tikzexternalenable}
\makeatother

\usepackage{microtype}
\usepackage{intcalc}
\usepackage{nicefrac}
\usepackage{pdflscape}
\usepackage{afterpage}
\usepackage{array}
\usepackage{authblk}
\usepackage{subcaption}

\usepackage{cases}
\usepackage{cleveref}
\usepackage[nocompress]{cite}

\usepackage{pgffor}
\usepackage{pgfplotstable}
\pgfplotsset{compat=1.13}
\usepackage{pgfplots}
\usepackage{tikz}
\usetikzlibrary{matrix}
\usepgfplotslibrary{groupplots}
\usepgfplotslibrary{external}

\tikzset{external/only named=true}
\tikzset{%
  external/system call={%
    lualatex \tikzexternalcheckshellescape%
    -interaction=batchmode -jobname "\image" "\texsource"%
  }%
}

\usepackage{custom-commands-private}

\newcolumntype{L}{>{$\displaystyle}l<{$}}

\theoremstyle{plain}
\newtheorem{theorem}{Theorem}[section]
\newtheorem{lemma}[theorem]{Lemma} 
\newtheorem{proposition}[theorem]{Proposition}
\newtheorem{problem}[theorem]{Problem}

\newtheorem{definition}{Definition}[section]

\graphicspath{{figures/}}

\begin{document}
\title[SG-HDG schemes]{A Scharfetter-Gummerl stabilization scheme for HDG approximations of convection-diffusion problems
\vskip1em
\footnotesize{\textnormal{STEFANO\hspace{3pt}PIANI,\hspace{3pt}LUCA\hspace{3pt}HELTAI,\hspace{3pt}AND\hspace{3pt}WENYU\hspace{3pt}LEI}}
}

\author[S.~Piani]{Stefano Piani}
\address{SISSA - International School for Advanced Studies, Via Bonomea 265, 34136 Trieste, Italy}
\email{\{stefano.piani, luca.heltai, wenyu.lei\}@sissa.it}


\begin{abstract}
We present a Scharfetter-Gummel (SG) stabilization scheme for high order Hybrid
Discontinuous Galerkin (HDG) approximations of convection-diffusion problems.
The scheme is based on a careful choice of the stabilization parameters that is
used to define the numerical flux in the HDG method. We show that, in one
dimension, the SG-HDG scheme is equivalent to the Finite Volume method
stabilized with the Scharfetter--Gummel on the dual grid, for all orders of HDG schemes.
\end{abstract}

\keywords{HDG methods, drift-diffusion problems, the Scharfetter-Gummel scheme, numerical fluxes}

\subjclass{65N30, 65N12}
\maketitle


\section{Introduction}\label{sec:intro}

In this paper, we consider a hybridizable discontinuous Galerkin (HDG) approximation scheme for the following convection-diffusion problems:
\begin{problem}
  \label{pb:flux-mixed}
  Let $\Omega \subset \mathbb{R}^n$ be a bounded domain with Lipschitz boundary $\Gamma$.
  Given a source term $f \in \Lspace{\Omega}$, and a Dirichlet boundary data $u_0 \in \Hspace[1/2]{\Gamma}$,
  we want to find $(\curru, u) \in \Hdiv{\Omega}\times\Hspace{\Omega}$ such that
  \begin{equation} \label{eq:intro-problem}
   \begin{cases}
     \curru + \diffcoeff \grad{u}  - u \driftcoeff = 0, & \text{ in } \Omega \eqcomma \\
     \diver{\curru} = f,  & \text{ in } \Omega \eqcomma \\
     u = u_0, & \text{ on } \Gamma 
   \end{cases}
  \end{equation}
 in the $L^2(\Omega)$ sense. 
\end{problem}

This system models the static solution of several different physical phenomena where the flux $\curru$
of an unknown quantity $u$ can be described by the combination of two different effects: transport
and diffusion. A major difficulty when considering the discretization of such systems arises in convection dominated problems, i.e., when
the P\'eclet number (the ratio between $\driftcoeff$ and $\diffcoeff$) is large.

The Scharfetter-Gummel (SG) stabilization technique \cite{scharfetter1969large}
(or exponential fitting method \cite{de1955relaxation}) combined with the Finite Volume (FV) method are considered the state-of-the-art to approximate
such problems. Given a subdivision of $\Omega$, the SG numerical flux between two adjacent cells is approximated by solving exactly the one-dimensional problem of
\eqref{eq:intro-problem} between the control points of the two cells, assuming that all coefficients are constant (see also Section~\ref{sec:sg-stabilization} for more details).
This results in a conservative scheme (cf. \cite{BessemoulinChatard2012}). Particularly in one dimensional space, and for constant convection and diffusion coefficients, the SG scheme recovers the solution $u$
exactly on grid points.

In this work, we focus on using discontinuous Galerkin (DG) methods to resolve flux conservation across cells; we refer to \cite{Arnold2002} for a unified
discussion. Our numerical method is based on HDG methods introduced by \cite{cockburn2009unified}, and used for systems of the same
nature as \Cref{eq:intro-problem} in \cite{nguyen2009implicit, nguyen2009implicitnonlinear, qiu2016hdg, chen2019hdg}. Our goal is to apply the same idea of the SG stabilization strategy to higher order approximations based on HDG schemes. One possibility is to exploit a Slotboom change of variable $\widetilde u = u \exp({-V/\diffcoeff})$, where $V$ is a potential field
such that $\driftcoeff = -\nabla V$. Various numerical methods then focus on the modified problem with respect to $\widetilde u$
including hybrid methods \cite{brezzi1989numerical, brezzi1989two, holst2008priori} originally introduced by \cite{arnold1985mixed}
(see also   \cite{LH-hdg} for another extension along this direction), edge averaged approaches
\cite{xu1999monotone, lazarov2012exponential}, and using exponential basis functions
\cite{angermann2005multidimensional,sacco1998finite}.

Similarly to what happens in other hybrid methods \cite{brezzi1989numerical, brezzi1989two}, one of the advantages of \HDG schemes is that one can
exploit hybridization (or static condensation) to eliminate all degrees of freedom defined on cells, resulting in high order finite element schemes with
a very advantageous ratio between accuracy and number of degrees of freedom.

In this paper, following the idea from the exponential fitting scheme,
we present a new \HDG stabilization strategy based on the
local HDG (L-HDG) scheme proposed by \cite{nguyen2009implicit}. Our scheme is inspired by the observation that the vertex-centered Scharfetter-Gummel scheme
solves $u$ exactly at grid points when $\diffcoeff$ and $\driftcoeff$ are constants, and we adjust the stabilization parameter in the H-LDG scheme so that it satisfies the same property in one dimension. Our main results in Theorem~\ref{thm:delta-k}, shows that such stabilization parameter exists using arbitrary degrees of polynomials. That is given a uniform subdivision of $\Omega$ in one dimensional space, with a proper selection of the stabilization parameter depending on the degree of polynomials, the mesh size and the P\'eclet number, the H-LDG approximate solution coincides with the one obtained by the SG finite volume scheme on the grid points. Formulas for such stabilization parameter for the polynomial degree at most $4$ is provided in Table~\ref{table:tau-val}.

The rest of the paper is organized as follows.
In \cref{sec:stabilization-flux-reconstruction} we introduce some notations for the H-LDG schemes as well as the SG finite volume methods. The SG stabilization based on the HDG methods is discussed in Section~\ref{sec:sg-hdg}. Our main results, the existing proof for such numerical scheme and tables the stabilization parameter with respect to the polynomial degree are also provided in this section. In Section~\ref{sec:num-examples}, we discuss some numerical experiments to justify our findings.

\section{Stabilization and flux reconstruction}
\label{sec:stabilization-flux-reconstruction}

\subsection{The \texorpdfstring{\HDG}{HDG} method}\label{subsec:hdg-method}

From what follows, we assume that the domain $\Omega$ is polytope.
Let $\{\triangulation_h\}_{h>0}$ be a family of quasi-uniform subdivisions of $\Omega$ made of
simplices with maximum size $h$. This means that which $h_T$ denoting the size of cell $T\in \mathcal T_h$ and $\rho_T$ denoting the size largest ball contained in $T$, there holds that for all $T\in \mathcal T_h$,
\[  
    h_T \le c\rho_T \le Ch 
\]
with the constants $c$ and $C$ independent of $T$.
We also denote $\faces_h$ the collection of faces of
$\triangulation_h$ and subdivide $\faces_h$ with
\begin{equation}
  \faces_h = \intfaces_h \cup \extfaces_h \eqcomma
\end{equation}
where $\intfaces_h$ and $\extfaces_h$ are the set of the interior and boundary faces, respectively. For convenience, we shall remove the subscript $h$ in the rest of the paper.

Given a non-negative integer $k$ and a cell $\cell\in\triangulation$, denote
$\polycell{k}{\cell}$ to be the Lagrange finite element space in \cell of degree at most $k$.
Set the finite element space
\begin{equation}
  \polyspace{k}{\triangulation} \Def \{
      v \in \Lspace{\Omega} : v|_\cell\in\polycell{k}{\cell} \text{ for }\cell\in\triangulation
  \} \eqdot
\end{equation}
Similarly, for each face $\face\in\faces$, we define $\polyface{k}{\face}$ to be the Lagrange
finite element space in \face of degree at most $k$ and
\begin{equation}
  \polytrace{k}{\faces} \Def \{
      \trace{v} \in \Lspace{\faces} : \trace{v}|_\face \in \polyface{k}{\face}
      \text{ for } \face \in \faces
  \} \eqdot
\end{equation}
Given $g \in \Lspace{\partial \Omega}$, denote $\polytrace[g]{k}{\faces}$ be an affine subspace of
$\polytrace{k}{\faces}$ so that for each $\face \in \extfaces$, the function
$\trace{v} \in \polytrace[g]{k}{\faces}$ satisfies that
\[
    \trace{v}|_F = \pi_\face g ,
\]
where $\pi_\face$ denotes the orthogonal projection onto $\polytrace{k}{\faces}$. Define the inner-products with respect to $\mathcal T$ and $\mathcal F$ by
\[
    (\cdot,\cdot) := \sum_{T\in\mathcal T} (\cdot,\cdot)_T
    \quad\text{and}\quad
    \langle\cdot,\cdot\rangle := \sum_{T\in\mathcal T} \langle\cdot,\cdot\rangle_{\partial T},
\]
where $(\cdot,\cdot)_T$ and $\langle\cdot,\cdot\rangle_{\partial T}$ are the $L^2$ inner products on $T$ and $\partial T$, respectively.

The \HDG[k] discretization of \eqref{pb:flux-mixed} reads: find $(u^h,\curru[h], \traceu[h]) \in
\polyspace{k}{\triangulation} \times [\polyspace{k}{\triangulation}]^n \times
\polytrace[u_0]{k}{\faces}$ satisfying that for all $(\stest, \vtest,\ttest) \in
\polyspace{k}{\triangulation} \times [\polyspace{k}{\triangulation}]^n \times
\polytrace[0]{k}{\faces}$,
\begin{equation}\label{eq:problem2_hdg_scheme}
\begin{cases}
  \Lprod{\curru[h]}{\vtest}
      - \Lprod{u^h}{\diver*{\diffcoeff \vtest}}
      - \Lprod{u^h \driftcoeff}{\vtest}
      + \Lbprod{\traceu[h]}{\diffcoeff \vtest \cdot \normal} = 0 \eqcomma \\
  - \Lprod{\curru[h]}{\nabla \stest} + \Lbprod{\curru*[h] \cdot \normal}{\stest}
      = \Lprod{f}{\stest} \eqcomma \\
  \Lbprod{\curru*[h] \cdot \normal}{\mu} = 0 \eqcomma
\end{cases}
\end{equation}
with the following numerical flux on \faces,
\begin{equation}\label{eq:tau}
  \curru*[h] \cdot \normal \Def \curru \cdot
      \normal + \tau (u - \traceu) \eqdot
\end{equation}
where $\tau$ is a positive function defined on \faces. Usually $\tau$ is in order $O(1)$ with respect to $h$. We also refer to \cite{nguyen2009implicit} for the well-posedness of discrete system in accordance with the above numerical settings. 

\subsection{The SG stabilization}
\label{sec:sg-stabilization}

The SG stabilization technique for Problem~\eqref{pb:flux-mixed} is usually applied for the simulation of the charge transportation in semiconductor devices, (\cite{scharfetter1969large}). Let us denote two adjacent cells with $T_1$ and $T_2$ and denote $F$ their shared face. We also set $\mathbf v_F$ the vector pointing from the center of $T_1$ to the center of $T_2$. The SG scheme is based on the observation that if $\diffcoeff$ and $\driftcoeff$ are constants, we can solve one dimensional problem \eqref{pb:flux-mixed} exactly along $\mathbf v_F$ and the resulting flux can be understood as the numerical flux on $F$. In practice, we consider the averages of $\diffcoeff$ and $\driftcoeff$ on $\mathbf v_F$ and denote them with $\diffcoeff_F$ and $\driftcoeff_F$. The the SG finite volume scheme reads: find a piecewice constant function $v$ on $\mathcal T$ so that
\begin{equation}\label{eq:sg}
    \sum_{T\in \mathcal T} \sum_{F\in \partial T}\int_F \mathbf J_v^{SG}\cdot\nu \, \mathrm{d}\gamma
     = \sum_T \int_T f \, \mathrm{d} \mathbf x ,
\end{equation}
where the numerical flux $\mathbf J_u^{SG}$ defined on each face $F$. Denoting $v_i$ the values of $v$ on $T_i$ for $i=1,2$, $\mathbf J_u^{SG}$ is given by
\begin{equation}\label{eq:fvm-sg}
  \int_{\face} \mathbf J_v^{SG} \cdot \normal\, \mathrm{d}\gamma
      := \frac{\measure{\face}\diffcoeff_{\face}}{l_{\face}} \left(\!
          \B{-\frac{\driftcoeff_{\face} \cdot \mathbf{v}_{\face}}{\diffcoeff_{\face}}} v_{1}
          - \B{\frac{\driftcoeff_{\face} \cdot \mathbf{v}_{\face}}{\diffcoeff_{\face}}} v_{2}
      \right) \eqcomma
\end{equation}
with $\mu(F)$ and $l_F$ denoting the measures of $F$ and $\mathbf v_F$, respectively and the function $B(\cdot)$ denotes the \emph{Bernoulli function}
\begin{equation} \label{eq:Bernoulli-def}
  \B{x} \Def \frac{x}{\mathrm{e}^x - 1} \eqdot
\end{equation}

Let us end this section by explicitly writing down the linear system of \eqref{eq:sg} in the one dimensional space. To this end, let $\Omega=(0,1)$ and consider the partition of $\Omega$ with the grid points $0=x_0<x_1<\ldots<x_N = 1$. For $i=1,\ldots,N$, we denote $v_i$ the value of the approximation $v$ in \eqref{eq:sg} in the interval $I_i = (x_{i-1},x_i)$. For $i=1,\ldots,N-1$, we also denote $s_i$ distance between the centers of $I_i$ and $I_{i}$ and set $s_0=(x_0+x_1)/2$ and $s_{N+1}=(x_N-x_{N-1})/2$. So the discrete system \eqref{eq:sg} becomes for $i=1,\ldots,N$,
\begin{equation}\label{eq:sg_system}
  \begin{multlined}[.98\displaywidth]
    - \frac{\diffcoeff}{\dhh[i - 1]} \B{- \frac{\driftscalar \dhh[i-1]}{\diffcoeff}} \fvmu{i-1} +
        \left[
          \frac{\diffcoeff}{\dhh[i - 1]} \B{\frac{\driftscalar \dhh[i-1]}{\diffcoeff}}
          + \frac{\diffcoeff}{\dhh[i]} \B{- \frac{\driftscalar \dhh[i]}{\diffcoeff}}
        \right] \fvmu{i} \\[10pt] -
    \frac{\diffcoeff}{\dhh[i]} \B{\frac{\driftscalar \dhh[i]}{\diffcoeff}} \fvmu{i + 1}
        = (x_i-x_{i-1}) f_i ,
  \end{multlined}
\end{equation}
where $f_i$ denotes the average of $f$ in $(x_{i_1},x_i)$. It is worth noting that for piecewise constant coefficients and right hand side data, the numerical scheme \eqref{eq:fvm-sg} approximates the solution $u$ exactly on the centers of cells in $\Omega$, denoted by $\{x_i\}$, i.e.
\[ 
    u(x_i) = v_i . 
\]





\section{Scharfetter--Gummel stabilization for \texorpdfstring{\HDG}{HDG} methods}\label{sec:sg-hdg}
In what follows, we assume that $\diffcoeff$, $\driftscalar$,
and $f$ are constants. For simplicity, we further assume that $\Omega$ is a unit interval, i.e $\Omega=(0,1)$. Our goal in this section is to find a suitable stabilization parameter $\tau$ so that the HDG trace approximation in \eqref{eq:problem2_hdg_scheme}, i.e. $\hat u^h$ on the skeleton of $\mathcal T$, coincides with the SG approximation $v$ defined on a dual (or staggered) grid of $\mathcal T$. Let us first provide the definition of the dual grid.

\begin{definition}
  Let \triangulation be a uniform triangulation of $\Omega$ with $N$ cells and let
  $\{x_i\}_{i=0}^{N}$ be a set of all its faces, so that $x_0 \Def 0$,
  $x_{N} \Def 1$ and for every $i \in \{1, \ldots, N\}$,
  \begin{equation}
   x_{i} - x_{i - 1} = \nicefrac{1}{N} \eqdot
  \end{equation}
  Let $\left\{y_i\right\}_{i=0}^{N + 1}$ be a collections of points satisfying that
  \begin{equation}
    y_0 \Def x_0, \quad
    y_{N+1} \Def x_N,\quad\text{and}\quad
    y_i \Def \frac{x_{i} + x_{i - 1}}{2}\text{ for } i \in \{1, \ldots, N \} .
  \end{equation}
  We call the \emph{dual triangulation} \dualt of \triangulation the collection of cells
  $\dcell_i = \left( y_{i}, y_{i+1} \right)$ for $i=0,\ldots,N$.
\end{definition}

Clearly, $\{x_i\}_{i=1}^{N - 1}$ are the centers of the cells $\dcell_i$ and we will approximate \eqref{eq:intro-problem} with the SG scheme on the dual grid $\mathcal T^d$. In order to simplify our argument, we say that $x_0$ and $x_N$ are also the
centers of $\dcell_0$ and $\dcell_{N}$, respectively. 


\begin{definition}\label{def:dual}
 Let \triangulation be uniform with mesh size $h=\tfrac1N$. We also let
  $\tracemd{i} \Def \{\traceu[h](x_i)\}_{i=0}^{N}$ with $\traceu[h]$ denoting the \HDG approximation on the trace according to \eqref{eq:problem2_hdg_scheme}.
  Set $\fvmu{}^h$ to be a SG approximation on \dualt based on \eqref{eq:sg} and set $\fvmu{i} =
  \fvmu{}^h(x_i)$. We say that the \HDG method is \emph{\equivalent} to the finite
  volume method if for every $i \in \{0, \ldots, N\}$,
  \[  \tracemd{i} = \fvmu{i} \eqdot \]
\end{definition}

In this section, we shall show the following main result.
\begin{theorem}\label{thm:delta-k}
For every degree $k > 0$, there exists a value $\tau_k$ such that the \HDG[k]
method with the stabilization parameter $\tau= \tau_k$ is \equivalent to the
Scharfetter--Gummel scheme.
\end{theorem}

The idea of the proof is to investigate the linear system for $\hat u^h$ which can be obtained by the static condensation. To this end, for each cell $\cell \in \triangulation$, denote
$\{\phi_i\}_{i=0}^{k}$ the set of shape functions in $\polycell{k}{\cell}$. For the
approximation $u^h$ in $\cell$, we set
$u^h \Def \sum_{i=0}^{k} u_i\phi_i$ with the coefficient vector $\mathbf u :=(u_0,\ldots,u_k)^T$.
Similarly, we set the approximation of the current $\curru[h] := \sum_{i=0}^{k} J_i\phi_i$
for some coefficient vector $\mathbf J := (J_0,\ldots, J_k)^T$.

For our proof, we need to introduce some constants that identify the properties of our problem.
First of all, we define the \emph{mesh Peclet number} as
\begin{equation}\label{eq:peclet}
  \peclet \Def \frac{\driftscalar h}{\diffcoeff} \eqdot
\end{equation}
This definition is analogous to the one that can be found, for example, in \cite{Quarteroni2017},
beside a factor two. Indeed, we usually have the following definition:
\begin{equation} \mpeclet \Def \frac{\driftscalar h}{2 \diffcoeff} \end{equation}
In our case, the constant 2 would increase the complexity of the computations and, therefore, we
omit it.
Finally, we define the constant
\begin{equation} \dtau \Def \frac{\tau h}{\diffcoeff} \eqdot \end{equation}
We also denote $\Lambda = \{c_{i,j}\}_{i,j=0}^N$ the system matrix for $\hat u^h$ and $\mathbf{r}$
the right hand side vector. The static condensation indicates that on each cell $T$, $\mathbf u$ and $\mathbf J$ are functions of the boundary values $\hat u^h$. According to transmission condition (the third equation in \eqref{eq:problem2_hdg_scheme}), we obtain that for $|i-j|>1$,
$c_{i,j}=0$.  The following lemma shows that for $|i-j|\le 1$ , $c_{i,j}$ is a function of
$\dtau$ and the Peclet number $\peclet$.

\begin{lemma}\label{c-coeffs}
The matrix $\Lambda$ is a tridiagonal Toeplitz matrix
\begin{equation}\label{eq:Lambda}
  \Lambda = %
  \frac{\diffcoeff}{h}\begin{pmatrix}
    c_2 & c_3 \\
    c_1 & c_2 & c_3  \\
        & c_1 & \ddots & \ddots \\
        &     & \ddots & \ddots & c_3 \\
        &     &        & c_1     & c_2
  \end{pmatrix}
\end{equation}
whose coefficients $c_1$, $c_2$ and $c_3$ depend only on $\dtau$ and $\peclet$. Moreover, there
exists a coefficient $r = r\left(\dtau, \peclet\right)$ such that the
right hand side vector $\mathbf{r}$ is a constant vector whose entries are all equal to
\begin{equation} (\mathbf{r})_i = h f r \eqdot \end{equation}
\end{lemma}
\begin{proof}
  Based on \eqref{eq:problem2_hdg_scheme} we write the local discrete system on
  $\cell\in\triangulation$ by $\mathbf A_\cell \mathbf X_\cell = \mathbf b_\cell$. Here
  \begin{equation}\label{eq:local-blocks}
  \mathbf X_\cell = \begin{pmatrix}
                        \mathbf J \\ \mathbf u\\
                    \end{pmatrix},
  \quad
  \mathbf b_\cell = \begin{pmatrix}
                        \mathbf b_1 \\ \mathbf b_2\\
                    \end{pmatrix}
  \text{ and }\quad
  \mathbf{A}_\cell = \begin{pmatrix}
                      \mathbf{A}_{11} && \mathbf{A}_{12} \\
                      \mathbf{A}_{21} && \mathbf{A}_{22} \\
                     \end{pmatrix}
  \end{equation}
  where
  \begin{itemize}
    \item $\mathbf{A}_{11} = \{\Lprod{\phi_j}{\phi_i}\}_{i,j=0}^{k}$. So the elements of this
    matrix scale linearly with $h$.
    \item $\mathbf{A}_{12} = \{- \Lprod{\phi_{j}}{\diver*{\diffcoeff \phi_i}}
    -\Lprod{\phi_{j}\driftscalar}{\phi_i}\}_{i,j=0}^{k}$. So the terms that appear inside this
    matrix are of the form $- k_1 \diffcoeff - k_2 \driftscalar h$, with $k_1, k_2 \in \mathbb{R}$ not depending on any parameter of the problem ($\tau$, $\diffcoeff$, $\driftscalar$ or
    $h$).
    \item $\mathbf{A}_{21} = \{\Lprod{\phi_j}{\nabla\phi_i} +
    \Lbprod{\phi_j \cdot \normal}{\phi_i}\}_{i,j=0}^{k}$. These elements do not depend on
    $\diffcoeff$, $\driftscalar$ or $h$.
    \item $\mathbf{A}_{22} = \{\Lbprod{\tau\phi_j}{\phi_i}\}_{i,j=0}^{k}$. In fact,
    $(\mathbf{A}_{22})_{00} = (\mathbf{A}_{22})_{kk} = \tau$ and all the other entries are zero.
    \item $\mathbf b_1 =\{-\Lbprod{\traceu[h]}{\diffcoeff \phi_i \cdot \normal}\}_{i=0}^k$.
    \item $\mathbf b_2 = \{\Lbprod{\tau\traceu[h]}{\phi_i} + \Lprod{f}{\phi_i}\}_{i=0}^k$.
  \end{itemize}

  Now we proceed with the following change of variables
  \begin{equation}\label{eq:hdgk-change-var}
    \mathbf v \Def \diffcoeff \mathbf u \eqcomma \qquad \qquad
    \tracev^h \Def \diffcoeff \traceu[h] \eqcomma \qquad \qquad
    \mathbf q \Def h\mathbf J \eqcomma
   \end{equation}
  and set $\widetilde{\mathbf X}_T= (\mathbf q,\mathbf v)^T$ so that
  \eqref{eq:local-blocks} becomes
  \begin{equation}\label{eq:rescaled-local-system}
    \begin{pmatrix}
        \mathbf{A}_{11}/h && \mathbf{A}_{12}/\alpha \\
        \mathbf{A}_{21} && h\mathbf{A}_{22}/\alpha \\
    \end{pmatrix}
    \widetilde{\mathbf X}_T
    = \begin{pmatrix}
        \mathbf b_1 \\ h\mathbf b_2\\
      \end{pmatrix}
  \end{equation}
  We denote the left hand side matrix above to be $\widetilde{\mathbf A}$ with block $\widetilde {\mathbf A}_{ij}$ for $i,j=1,2$. So the matrix
  $\widetilde{\mathbf{A}}_{11}$ loses its dependency on $h$.
  The elements of $\widetilde{\mathbf{A}}_{12}$, instead, become linear functions respect of
  $\peclet$. The matrix $\widetilde{\mathbf{A}}_{21} = \mathbf A_{21}$ while the matrix
  $\widetilde{\mathbf{A}}_{22}$ is a matrix function of $\delta$.

  Now we want to apply the static condensation by combining the rescaled system
  \eqref{eq:rescaled-local-system} and the third equation from \eqref{eq:problem2_hdg_scheme}, i.e.
  the transmission condition.
  Recalling that $\{\traceu_i\}_{i=0}^N$ are the unknowns on the trace, for each $i$, let $\cell_l$
  and $\cell_r$ be the left and right adjacent cells, i.e $l=i-1$ and $r=i$.  We also
  denote $\phi_{l,k}$ (or $\phi_{r,0}$) the only one
  shape basis function which is nonzero on the left (or right) boundary of the cell and denote
  $(J_r^{(T_l)}, u_r^{(T_l)})$ (or $(J_l^{(T_r)}, u_l^{(T_r)})$) the corresponding coefficients for $(\mathbf J,\mathbf u)$.
  When it is important to point out on which cell a coefficient is computed, we will
  indicate it with a superscript inside two parenthesis. Instead, we will avoid it when the coefficient is not
  cell-dependent, i.e. when the same computation can be performed on any cell obtaining the same results. Finally, we set $\mathbf{e}_l$
  (or $\mathbf{e}_r$) the canonical vector of $\mathbb{R}^{k+1}$ associated with $\phi_{l,0}$ (or
  $\phi_{r,k}$). Whence, the discrete transmission condition becomes
  \[
  (\tau u_r^{(\cell_l)} + J_r^{(\cell_l)}) + (\tau u_l^{(\cell_r)} - J_l^{(\cell_r)}) - 2 \tau \traceu_i = 0 \eqdot
  \]
  Using the change of variables in \eqref{eq:hdgk-change-var}, we obtain that
  \begin{equation} \label{eq:v-q-vhat}
    (\dtau v_r^{(\cell_l)} + q_r^{(\cell_l)}) + (\dtau v_l^{(\cell_r)} - q_l^{(\cell_r)})  - 2 \dtau \tracev_i = 0 \eqdot
  \end{equation}
  Letting
  \begin{equation} \label{eq:pl-defi}
    \mathbf{p}_r \Def (\mathbf{e}_r, \dtau \mathbf{e}_r)^T\eqcomma \qquad \textrm{and} \qquad \mathbf{p}_l \Def (-\mathbf{e}_l, \dtau \mathbf{e}_l)^T \eqcomma
  \end{equation}
  and based on the rescaled system \eqref{eq:rescaled-local-system}, we write
    \begin{equation}\label{eq:def-coeff-trace-system-right}
    \dtau v_r^{(\cell_l)} + q_r^{(\cell_l)} = \mathbf{p}_r^T \widetilde{\mathbf X}_T = \mathbf{p}_r^T
        \widetilde{\mathbf{A}}^{-1} \begin{pmatrix}
            \mathbf b_1^{(\cell_l)} \\ h\mathbf b_2^{(\cell_l)}\\
      \end{pmatrix}
  \end{equation}
  Similarly,
  \begin{equation}\label{eq:def-coeff-trace-system-left}
    \dtau v_l^{(\cell_r)} - q_l^{(\cell_r)} = \mathbf{p}_l^T \widetilde{\mathbf X}_T = \mathbf{p}_l^T
        \widetilde{\mathbf{A}}^{-1} \begin{pmatrix}
            \mathbf b_1^{(\cell_r)} \\ h\mathbf b_2^{(\cell_r)}\\
      \end{pmatrix} .
  \end{equation}
  Next we want to investigate the dependency of $\delta$ and $\peclet$ for the vector $\mathbf{b_{\cell}}$. Define the
  vector $\boldphi\in \mathbb R^{k+1}$ such that for $i=0,\ldots,k$,
  \[ \boldphi_i \Def \frac1h\int_{\cell} \! \phi_i \, \mathrm{d}x . \]
  According to the definition of $\mathbf b_T$ in \eqref{eq:local-blocks}, we can derive that
  \begin{align}
    \mathbf{b}_1^{(\cell_l)} & = \tracev_{i-1} \mathbf{e}_l - \tracev_{i} \mathbf{e}_r\eqcomma &
    \mathbf{b}_1^{(\cell_r)} & = \tracev_{i} \mathbf{e}_l - \tracev_{i+1} \mathbf{e}_r \eqcomma \label{eq:b-one-def}\\
    \mathbf{b}_2^{(\cell_l)} & =
        h f \boldphi + \frac{\dtau}{h} \tracev_{i - 1} \mathbf{e}_l + \frac{\dtau}{h} \tracev_{i}
        \mathbf{e}_r \eqcomma &
    \mathbf{b}_2^{(\cell_r)} & =  h f \boldphi + \frac{\dtau}{h} \tracev_{i} \mathbf{e}_l + \frac{\dtau}{h} \tracev_{i + 1}
        \mathbf{e}_r . \label{eq:b-two-def}
  \end{align}
  We combine to above two equations by setting (for both cells $\cell_l$ and $\cell_r$):
  \begin{equation}\label{eq:b-vectors-def}
   \mathbf{b}_l \Def \begin{pmatrix} \mathbf{e}_l \\ \dtau \mathbf{e}_l \end{pmatrix}
   \quad\text{ and } \quad
   \mathbf{b}_r \Def \begin{pmatrix} - \mathbf{e}_r \\ \dtau \mathbf{e}_r \end{pmatrix}
  \end{equation}
  so that
  \begin{align}
    \begin{pmatrix}
      \mathbf b_1^{(\cell_l)} \\ h\mathbf b_2^{(\cell_l)}\\
    \end{pmatrix} =
      h^2 f \begin{pmatrix} \mathbf{0} \\ \boldphi \end{pmatrix} +
      \tracev_{i - 1} \mathbf{b}_l + \tracev_{i} \mathbf{b}_r \label{eq:blr} \eqcomma \\[4pt]
    \begin{pmatrix}
      \mathbf b_1^{(\cell_r)} \\ h\mathbf b_2^{(\cell_r)}
    \end{pmatrix} =
      h^2 f \begin{pmatrix} \mathbf{0} \\ \boldphi \end{pmatrix} +
      \tracev_{i} \mathbf{b}_l + \tracev_{i + 1} \mathbf{b}_r \label{eq:blr-bis} \eqdot \\
  \end{align}
  Inserting \eqref{eq:blr} and \eqref{eq:blr-bis} into \eqref{eq:def-coeff-trace-system-left}
  and \eqref{eq:def-coeff-trace-system-right} to write
  \begin{equation}
    \dtau v_r^{(\cell_l)} + q_r^{(\cell_l)} = h^2 f \mathbf{p}_r^{T} \widetilde{\mathbf{A}}^{-1} \begin{pmatrix}
        \mathbf{0} \\ \boldphi\\
      \end{pmatrix}  + \tracev_{i - 1} \mathbf{p}_r^{T} \widetilde{\mathbf{A}}^{-1} \mathbf{b}_l +
      \tracev_{i}\mathbf{p}_r^{T} \widetilde{\mathbf{A}}^{-1} \mathbf{b}_r \eqdot
  \end{equation}
  and
  \begin{equation}
    \dtau v_l^{(\cell_r)} - q_l^{(\cell_r)} = h^2 f \mathbf{p}_l^{T} \widetilde{\mathbf{A}}^{-1} \begin{pmatrix}
        \mathbf{0} \\ \boldphi\\
      \end{pmatrix}  + \tracev_{i} \mathbf{p}_l^{T} \widetilde{\mathbf{A}}^{-1} \mathbf{b}_l +
      \tracev_{i + 1} \mathbf{p}_l^{T} \widetilde{\mathbf{A}}^{-1} \mathbf{b}_r
  \end{equation}

  Finally, we apply the above two equations in \eqref{eq:v-q-vhat} and combing coefficients with
  respect to $\{\hat{v}_i\}$. Define these coefficients by
  \begin{equation} \label{eq:r-def}
    r = r(\dtau, \peclet) \Def - \left(\mathbf{p}_l + \mathbf{p}_r \right)^{T}
              \widetilde{\mathbf{A}}^{-1} \begin{pmatrix}
        \mathbf{0} \\ \boldphi\\
      \end{pmatrix}
  \end{equation}
  \begin{equation}\label{eq:c123}
  \begin{aligned}
   c_1 &= c_1(\dtau, \peclet) \Def \mathbf{p}_r^T\widetilde{\mathbf{A}}^{-1} \mathbf{b}_l\eqcomma\\
   c_2 &= c_2(\dtau, \peclet) \Def \mathbf{p}_l^{T}\widetilde{\mathbf{A}}^{-1} \mathbf{b}_l +
     \mathbf{p}_r^{T} \widetilde{\mathbf{A}}^{-1} \mathbf{b}_r - 2 \delta \eqcomma\\
   c_3 &= c_3(\dtau, \peclet) \Def \mathbf{p}_l^T \widetilde{\mathbf{A}}^{-1} \mathbf{b}_r \eqcomma
   \end{aligned}
  \end{equation}
  and write
  \begin{equation}
    c_1\tracev_{i-1} + c_2 \tracev_{i} + c_3 \tracev_{i + 1} = h^2 f r \eqcomma
  \end{equation}
  The previous equation can be rewritten respect to $\traceu_i$:
  \begin{equation}\label{eq:condensed-system-explicit}
    \diffcoeff \left(c_1 \traceu_{i-1} + c_2 \traceu_{i} + c_3 \traceu_{i + 1}\right) = h^2 f r
    \eqdot
  \end{equation}
  Noting that the previous equation has been manipulated by multiplying the original
  equation with a factor $h$, we thus obtain the system matrix \eqref{eq:Lambda}
  whose entries $c_1$, $c_2$ and $c_3$ functions of $\delta$ and $\peclet$. The proof is complete.
\end{proof}

In the next lemma, we simplify the right hand side of the global system by showing that the constant $r$ in \eqref{eq:r-def}
does not depend on $\delta$, $\peclet$ and the polynomial degree $k$.

\begin{lemma}
  There holds that $r=-1$ in \eqref{eq:r-def} for all polynomial degree $k$.
\end{lemma}
\begin{proof}
Note that the right hand side of \eqref{eq:r-def} is an algebraic form that does not depend on the data $f$. Though $r$ is obtained from the discrete transmission condition \eqref{eq:problem2_hdg_scheme} between two adjacent cells (i.e. static condensation), we can actually reproduce such algebraic form in a simpler mesh setting.


Consider the model problem \eqref{eq:intro-problem} on  $\Omega = (0, h)$ associated with the homogeneous Dirichlet boundary condition. We shall approximate the solution using the HDG scheme \eqref{eq:problem2_hdg_scheme} with the mesh $\mathcal T$ that contains only one cell, namely $\mathcal T = \{\Omega\}$. Letting the data $f=1/h^2$, we can approximate $\mathbf J^h_u$ and $u^h$ by directly solving the local problem $\mathbf A_T\mathbf X_T = \mathbf b_T$ introduced by \eqref{eq:local-blocks}. Recalling that $\mathbf J$ and $u$ are the corresponding finite element coefficient vectors with dimension $k+1$, we follow from the change of variables in \eqref{eq:hdgk-change-var} as well as the rescaled local system \eqref{eq:rescaled-local-system} to get 
\begin{equation}
\begin{pmatrix}
      h\mathbf J \\ \mathbf \alpha \mathbf u\\
  \end{pmatrix} = 
  \begin{pmatrix}
      \mathbf q \\ \mathbf v\\
  \end{pmatrix}
   = \widetilde{\mathbf{A}}^{-1} \begin{pmatrix} \mathbf b_1 \\ h\mathbf b_2 \end{pmatrix}
  = \widetilde{\mathbf{A}}^{-1} \begin{pmatrix} \mathbf 0 \\ \boldphi \end{pmatrix} ,
\end{equation}
where for the last equality we used the fact that $\mathbf b_1 = 0$ due to the zero boundary condition and $\mathbf b_2 = hf\boldphi = \tfrac1h\boldphi$.
This leads to 
\begin{equation}\label{e:r-rep}
\begin{aligned}
    r &= - \left(\mathbf{p}_l + \mathbf{p}_r \right)^{T}
              \widetilde{\mathbf{A}}^{-1} \begin{pmatrix}
        \mathbf{0} \\ \boldphi\\
      \end{pmatrix} \\
      &= - \left[\begin{pmatrix}
        -\mathbf e_{0} \\ \delta \mathbf e_{0}\\
      \end{pmatrix}^T +\begin{pmatrix}
        \mathbf e_{k} \\ \delta \mathbf e_{k}\\
      \end{pmatrix}^T \right]
       \begin{pmatrix}
      h\mathbf J \\ \mathbf \alpha \mathbf u\\
  \end{pmatrix}
  = -h(\tau u_k + J_k + \tau u_0 - J_0) .
  \end{aligned}
\end{equation}
Here $\mathbf e_i$ denotes the canonical vector for the $(i+1)-$th component.

On the other hand, we choose the test function $v=1$ in the second equation of \eqref{eq:problem2_hdg_scheme} to get 
\[
 \Lbprod[\partial\Omega]{\curru[h] \cdot \normal + \tau\left(u^h - \traceu[h]\right)}{1} = \Lprod[\Omega]{f}{1} .
\]
Using coefficient vectors as well as $f=1/h^2$ to rewrite the above equation as
\[
    -J_0+\tau u_0 + J_k +\tau u_k = \frac1h .
\]
Combing \eqref{e:r-rep} with the above equation immediately implies that $r=-1$.

\end{proof}

Since $r$ is a constant, in order to show that there exists a value of $\tau$
for which the \HDG[k] method is \equivalent to the Scharfetter--Gummel scheme, we need to better
understand the structure of the coefficients $c_i$ defined in \Cref{c-coeffs} for $i=1,2,3$.

\begin{lemma}\label{thm:c-coeffs-relationships}
 Let $c_{1}$, $c_2$ and $c_{3}$ be as defined in \Cref{c-coeffs}. There hold
 \begin{enumerate}[(a)]
  \item $c_{1} + c_2 + c_{3} = 0$; \label{itm:c-coeff-sum}
  \item $c_{1}(\dtau, - \peclet) = c_{3} (\dtau, \peclet)$; \label{itm:c-coeff-even}
  \item $c_{3} - c_{1} = -\peclet$. \label{itm:difference}
 \end{enumerate}
\end{lemma}

\begin{proof}
  To prove the point~\ref{itm:c-coeff-sum}, it is enough to check that, when $f = 0$ and the
boundary conditions impose that $\traceu_{0} = 1$ and $\traceu_{N} = 1$, the solution of the
system~\eqref{eq:problem2_hdg_scheme} are three constants $u^h = 1$, $\curru[h] = \driftscalar$ and
$\traceu = 1$. 
Imposing that the constant vector $(1, 1, \ldots, 1)$ is a solution of the
homogeneous system~\eqref{eq:condensed-system-explicit} gives the thesis.

For what concerns the point~\ref{itm:c-coeff-even}, this can be proven by using the change of
variable $\tilde u(x) \Def u(l -x)$ and the symmetry of the problem.

Finally, let us consider the point~\ref{itm:difference}. For \HDG[0] methods, the statement can be
proven simply computing explicitly the values of $c_{1}$ and $c_{3}$ as functions of $\delta$ and
$\peclet$.
Therefore, here we will take into account only methods of degree greater or equal of 1.
We consider a domain $[-h, h]$ made of two cells of the same size, so that $\traceu[h]$ is defined
on the points $-h$, $0$ and $h$. We choose
\begin{equation}
 f \Def \frac{\driftscalar}{h}
\end{equation}
as a constant and we impose $\traceu[h](-h) = -1$ and $\traceu[h](h) = 1$. A solution of the
problem~\eqref{eq:problem2_hdg_scheme} is therefore
\begin{equation}
 u^h = \frac{x}{h} \eqcomma \qquad \qquad \curru^h = \frac{\driftscalar x - \diffcoeff}{h} \eqcomma
 \qquad \qquad \traceu[h](0) = 0 \eqdot
\end{equation}
Applying equation~\ref{eq:condensed-system-explicit} (divided by $r$), we get
\begin{equation}
 c_1 \traceu[h](-h) + c_2 \traceu[h](0) + c_3 \traceu[h](h) = - \frac{h
 \driftscalar}{\diffcoeff}
\end{equation}
which is exactly
\begin{equation}
 c_3 - c_1 = -\peclet
\end{equation}
\end{proof}

\begin{lemma}\label{thm:dtau-degree}
The coefficients $c_{1}$, $c_2$ and $c_{3}$ are rational functions of $\dtau$ and $\peclet$,
i.e for each index $i \in \{1, 2, 3\}$, there exist two polynomials $p_i(\dtau, \peclet)$ and
$d_i(\dtau, \peclet)$ such that
\[ c_i = \frac{p_i}{d_i} \eqdot \]
Moreover, for $i \in \{1, 3\}$, the degree respect to the variable $\dtau$ of $p_i$ and $d_i$ is
smaller or equal than 1.
\end{lemma}

\begin{proof}
 To prove this, we will use the Cramer's rule. Let us start by considering the determinant of the
 matrix $\widetilde{\mathbf{A}}$ defined in the proof of \Cref{c-coeffs}. It is clear
that the determinant is a polynomial in $\dtau$ and $\peclet$ and that the degree respect to
$\dtau$ must be less or equal than two (because $\dtau$ appears only on
the first and last element of $\widetilde{\mathbf{A}}_{22}$). Moreover, if $\dtau$ is 0 we have
that $\widetilde{\mathbf{A}}$ is singular. This is a well known result but it can also be easily
proven by noticing that the matrix $\widetilde{\mathbf{A}}_{21}$ is singular (because is a
projection of the space of polynomial of degree $k$ onto a space of degree $k - 1$) and
$\widetilde{\mathbf{A}}_{22}$ is identically 0 when $\dtau$ is 0. Therefore, there exists a linear
combination of the last $k + 1$ rows that is zero. Because of this, we have that there exists a
polynomial $d$ in $\peclet$ and $\dtau$ of degree less
or equal than 1 in $\dtau$ such that
\begin{equation}
 \mathrm{det}\left(\widetilde{\mathbf{A}}\right) = \dtau d \eqdot
\end{equation}

To compute
\begin{equation*}
  c_1 = \mathbf{p}_r^T\widetilde{\mathbf{A}}^{-1} \mathbf{b}_l \eqcomma
\end{equation*}
we need to evaluate the determinants of two matrices obtained by substituting the columns relative to
$v_r$ and $q_r$ with the vector $\mathbf{b}_l$ defined in \eqref{eq:b-vectors-def}. We denote
$\widetilde{\mathbf{A}}_{v_r}$ the matrix obtained by substituting into the matrix the
column relative to the unknown $v_r$ with the vector $\mathbf{b}_l$. In an analogous way, we define
also the matrix $\widetilde{\mathbf{A}}_{q_r}$. Recalling the definition of $\mathbf{p}_r$
given in~\eqref{eq:pl-defi}, we have that
\begin{equation}
 c_1 = \frac{%
     \dtau \mathrm{det}\big(\widetilde{\mathbf{A}}_{v_l}\big) +
           \mathrm{det}\big(\widetilde{\mathbf{A}}_{q_l}\big)
 }{%
   \mathrm{det}\big(\widetilde{\mathbf{A}} \big)%
 } \eqdot
\end{equation}

Now we compute $\mathrm{det}\big(\widetilde{\mathbf{A}}_{q_l}\big)$: in that case, there
are three entries of the matrix that depends on $\dtau$ (unless $k = 0$, but in this case a
trivial computation shows that $\mathrm{det}\big(\widetilde{\mathbf{A}}_{q_l}\big) = \delta(\peclet + 2)$);
two are in the block
$\widetilde{\mathbf{A}}_{22}$ and another one is the one introduced by the vector $\mathbf{b}_l$.
But two of this two entries are on the same row and therefore, by the Laplace expansion, we have
again that the degree must be smaller or equal than 2. Finally, if we impose $\dtau = 0$ we obtain
that the determinant is again 0. Indeed, we have previously shown that the zero vector can be
written as a linear combination of the last $k + 1$ rows of the matrix $\widetilde{\mathbf{A}}$
when $\dtau$ is zero. The same linear combination is also zero if applied on the last $k + 1$ rows
of $\widetilde{\mathbf{A}}_{q_l}$ because $\mathbf{b}_l$ contains only zeros in the last $k + 1$
entries (when $\dtau$ is 0). Therefore, as we did for the matrix $\widetilde{\mathbf{A}}$, we have
shown that there exist a polynomial $m_{q_l}$ in $\peclet$ and $\dtau$ such that
\begin{equation}
  \mathrm{det}\big(\widetilde{\mathbf{A}}_{q_l}\big) = \dtau m_{q_l} \eqdot
\end{equation}
and $\mathrm{deg}_{\dtau}(m_{q_l}) \leq 1$.

Finally, for what concerns $\widetilde{\mathbf{A}}_{v_l}$, we have again that there are only two
elements that contains $\dtau$ and, moreover, they are on the same row. This means that the degree
of the determinant of $\widetilde{\mathbf{A}}_{v_l}$ is a polynomial of degree smaller or equal
than 1 respect to $\dtau$. Taking into account that, when $\dtau$ is equal to 0, the last $k + 1$
rows of the matrix $\widetilde{\mathbf{A}}_{v_l}$ coincide with the last $k + 1$ rows of the matrix
$\widetilde{\mathbf{A}}$, we have that exist a polynomial $m_{v_l}$ that does not contains $\dtau$
(but only $\peclet$) such that
\begin{equation}
  \mathrm{det}\big(\widetilde{\mathbf{A}}_{v_l}\big) = \dtau m_{v_l} \eqdot
\end{equation}

Therefore, we can conclude that
\begin{equation}
 c_1 = \frac{\dtau m_{v_l} + m_{q_l}}{d} \eqcomma
\end{equation}
after having simplified by $\dtau$. This proves the thesis for $c_1$.

For $c_3$, the thesis holds because of point~\ref{itm:c-coeff-even} of
\Cref{thm:c-coeffs-relationships} and, finally, for $c_2$ we can use point~\ref{itm:c-coeff-sum} of
the same lemma.
\end{proof}

Now we have all the ingredients for exposing the proof of the main theorem of this section:

\begin{proof}[Proof of \Cref{thm:delta-k}]
 From \Cref{thm:dtau-degree} we know that $c_1$ can be written as
 \begin{equation}
  c_1 = \frac{l}{q} \eqcomma
 \end{equation}
 with $l$ and $q$ polynomials in $\peclet$ and $\dtau$ of degree respect to $\dtau$ smaller or
equal than 1.

Using point~\ref{itm:difference} of \Cref{thm:c-coeffs-relationships} we can affirm that
\begin{equation}
 c_3 = \frac{l - q\peclet}{q} \eqdot
\end{equation}

Given a polynomial $d(\peclet)$, we denote with $\swapped{d}\Def d(-\peclet)$. Using now
point~\ref{itm:c-coeff-even} \Cref{thm:c-coeffs-relationships}, we obtain that
\begin{equation} \frac{\swapped{l}}{\swapped{q}} = \frac{l - q\peclet}{q} \eqdot \end{equation}

Because of the fact that $q$ and $\swapped{q}$ have the same degree (respect to both $\peclet$ and
$\dtau$), we have that
\begin{equation} q = \swapped{q}\eqcomma \qquad \textrm{and} \qquad l = \swapped{l} + q\peclet \eqcomma \end{equation}
i.e., $q$ is an even polynomial (respect to $\peclet$) and its coefficients are the odd coefficients
of $l$ divided by $2 \peclet$.

Using \Cref{thm:dtau-degree}, we have just proved that there exist 4 \emph{even} polynomials in
$\mathbb{R}[\peclet]$ (that do not depend on $\dtau$) $q_0, q_1, s_0, s_1$ such that
\begin{equation}\label{eq:c1-coeff-structure}
  c_1 = \frac{(s_1 + q_1\peclet)\dtau + (s_0 + q_0\peclet)}{2(q_1\dtau + q_0)} \eqdot
\end{equation}

We define
\begin{equation}\label{eq:dtau-def}
 \dtau_k \Def - \frac{
     \mathrm{e}^\peclet(s_0 - q_0\peclet) - (s_0 + q_0\peclet)
   }{
     \mathrm{e}^\peclet(s_1 - q_1\peclet) - (s_1 + q_1\peclet)
   } \eqcomma
\end{equation}
and a trivial computation shows that replacing $\dtau$ with the value $\dtau_k$ inside
equation~\eqref{eq:c1-coeff-structure} we obtain
\begin{equation}
 c_1(\dtau_k, \peclet) = \frac{\peclet \mathrm{e}^\peclet}{\mathrm{e}^{\peclet} - 1} =
\B{-\peclet} \eqdot
\end{equation}

Therefore, beside the coefficient $-\frac{\alpha}{h}$, we have obtained the same coefficient of
equation~\eqref{eq:sg_system} for $\traceu_{i - 1}$. The same happens also for the other two terms
of the equation because of the points~\ref{itm:c-coeff-even} and \ref{itm:c-coeff-sum} of
\Cref{thm:c-coeffs-relationships}, and this concludes the proof.
\end{proof}

Using a CAS system, it is possible to compute explicitly the value of $\delta_k$ for a specific
degree $k$. Table~\ref{table:tau-val} shows these values up to degree 4. Instead, in
\autoref{fig:delta-k}, it is possible to see the plot of $\delta_k$ as a function of $\peclet$ for
$k \in \{0, 1, \ldots, 5\}$.

\begin{figure}[hbt!]
  \pgfplotstableread[col sep=comma]{data/delta_values.csv}\tauzdata
  \centering
  \tikzsetnextfilename{delta-vs-peclet}
  \begin{tikzpicture}
    \begin{axis}[xmode=log, ymode=log, axis x line=bottom, axis y line=left, xmin = 1e-2, xmax=1e3,
                 xlabel={$\peclet$}, legend pos=south east]
       \pgfplotsforeachungrouped \k/\c in {0/blue, 1/cyan, 2/green, 3/yellow, 4/orange, 5/red}{
          \edef\addanotherplot{
            \noexpand\addplot [\c] table [x={peclet}, y={delta\k}]{\noexpand\tauzdata};
            \noexpand\addlegendentry{\noexpand\(\noexpand\dtau_\k\noexpand\)}
          }
          \addanotherplot
       }
    \end{axis}
  \end{tikzpicture}
  \caption{The values of $\dtau_k$ for $k \in \{0, 1, \ldots, 5\}$.}
  \label{fig:delta-k}
\end{figure}
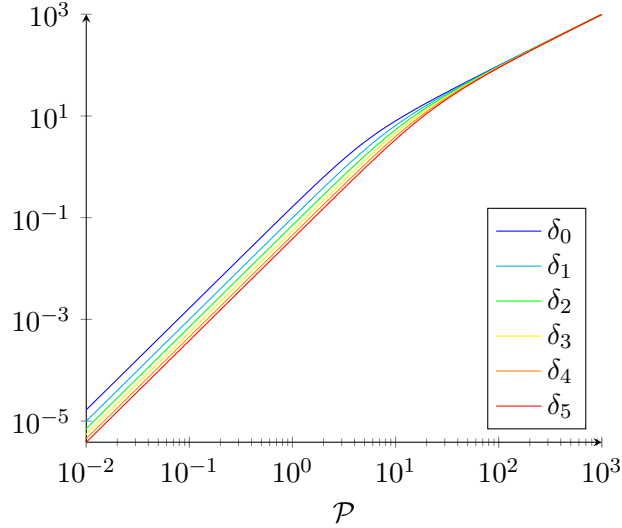

\afterpage{%
  \clearpage
  \thispagestyle{empty}
  \begin{landscape}
    \vspace*{\fill}
    \centering 
    \renewcommand{\arraystretch}{4}
    \begin{tabular}{L | L}
      \dtau_0 & - \frac{%
          \mathrm{e}^{\peclet}\left(-\peclet + 2 \right) - \left(\peclet + 2 \right)
        }{%
          \mathrm{e}^{\peclet} - 1
        }\\
      \dtau_1 & - \frac{%
          \mathrm{e}^{\peclet}\left(%
            \peclet^{2} - 6 \, \peclet + 12
          \right) - \left(%
            \peclet^{2} + 6 \,\peclet + 12
          \right)%
        }{%
          \mathrm{e}^{\peclet}\left(-\peclet + 2 \right) - \left(\peclet + 2 \right)
        }\\
      \dtau_2 & - \frac{%
          \mathrm{e}^{\peclet}\left(%
            -\peclet^{3} + 12 \, \peclet^{2} - 60 \, \peclet + 120
          \right) - \left(%
            \peclet^{3} + 12 \, \peclet^{2} + 60 \, \peclet + 120
          \right)%
        }{%
          \mathrm{e}^{\peclet}\left(%
            \peclet^{2} - 6 \, \peclet + 12
          \right) - \left(%
            \peclet^{2} + 6 \,\peclet + 12
          \right)%
        }\\
      \dtau_3 & - \frac{%
          \mathrm{e}^{\peclet}\left(%
            \peclet^{4} - 20 \, \peclet^{3} + 180 \, \peclet^{2} - 840 \, \peclet + 1680
          \right) - \left(%
            \peclet^{4} + 20 \, \peclet^{3} + 180 \, \peclet^{2} + 840 \, \peclet + 1680
          \right)%
        }{%
          \mathrm{e}^{\peclet}\left(%
            -\peclet^{3} + 12 \, \peclet^{2} - 60 \, \peclet + 120
          \right) - \left(%
            \peclet^{3} + 12 \, \peclet^{2} + 60 \, \peclet + 120
          \right)%
        }\\
      \dtau_4 & - \frac{%
          \mathrm{e}^{\peclet}\left(%
            -\peclet^{5} + 30 \, \peclet^{4} - 420 \, \peclet^{3} + 3360 \, \peclet^{2} - 15120 \,
                \peclet + 30240
          \right) - \left(%
            \peclet^{5} + 30 \, \peclet^{4} + 420 \, \peclet^{3} + 3360 \, \peclet^{2} + 15120 \,
                \peclet + 30240
          \right)%
        }{%
          \mathrm{e}^{\peclet}\left(%
            \peclet^{4} - 20 \, \peclet^{3} + 180 \, \peclet^{2} - 840 \, \peclet + 1680
          \right) - \left(%
            \peclet^{4} + 20 \, \peclet^{3} + 180 \, \peclet^{2} + 840 \, \peclet + 1680
          \right)%
        }
    \end{tabular}
    \captionsetup{type=table}
    \captionof{table}{
        Values of $\dtau_k$ for $k \in \{0, \ldots, 4\}$. %
        $\tau_k$ is defined as $\frac{\diffcoeff}{h}\dtau_k$%
    }
    \label{table:tau-val}
    \vspace*{\fill}
    \end{landscape}
    \clearpage
}

Before concluding this section, it is worth noticing that the values shown in
table~\ref{table:tau-val} are unique, i.e. there exists only one possible choice of $\tau$ that
minimize the error on the trace. This is a consequence of the following lemma.

\begin{proposition}\label{thm:unique-tau}
For a given degree $k$, there exists a \emph{unique} value $\tau_k$ so that the solution
$\traceu[h]$ on the trace of an \HDG[k] method defined in \eqref{eq:problem2_hdg_scheme} for $f =
0$ and every choice of the parameters $\diffcoeff$ and $\driftscalar$ and of the Dirichlet boundary
conditions $\traceu(0)$ and $\traceu(1)$, coincides with $\traceu$ on every point of the trace.
\end{proposition}

\begin{proof}
The existence has already been proven in the \Cref{thm:delta-k}. For the uniqueness, let us take a
point $x_i$ on the trace which is on the boundary between the cell $\cell_i$ and the cell $\cell_{i
+ 1}$. Because $\traceu[h]$ on the trace coincides with the analytical solution in every point of
the trace, we can restrict our method only on the two cells $\cell_i$ and $\cell_{i + 1}$ using the
values of the trace on $x_{i - 1}$ and $x_{i + 1}$ as Dirichlet boundary conditions. We denote with
$k_{i-1}$ and $k_{i + 1}$ a possible choice of the boundary conditions on the point $x_{i-1}$ and
$x_{i + 1}$ and with $k(\diffcoeff, \driftscalar, k_{i - 1}, k_{i + 1})$ the value of the exact
solution on the point $x_i$.

Then we have that for every choice of $k_{i-i}$ and $k_{i + 1}$.
\begin{equation} c_1 k_{i-1} + c_2 k + c_3 k_{i + 1} = 0 \eqcomma \end{equation}
where the coefficients $c_1$, $c_2$ and $c_3$ have been introduced in
\Cref{thm:c-coeffs-relationships}.

In particular, we can choose $k_{i + 1} = 0$ obtaining that
\begin{equation}\label{eq:unique-tau} \frac{c_2}{c_1} = - \frac{k_{i - 1}}{k} \eqdot \end{equation}
Only the left hand side of the previous equation depends on $\tau$; moreover, because of what we
have exposed in the proof of \Cref{thm:delta-k}, there exists 4 coefficients $s_1$,
$s_2$, $s_3$ and $s_4$ such that
\begin{equation} \frac{c_2}{c_1} = \frac{s_1 \tau + s_2}{s_3 \tau + s_4} \end{equation}
Therefore, respect to $\tau$, equation~\eqref{eq:unique-tau} admits only one solution and this ends
the proof.
\end{proof}

\section{Numerical examples} \label{sec:num-examples}
In this section we illustrate some experiments we performed related to the error of the \HDG
method applied on the equation
\[ \frac{\partial}{\partial x} \left(\beta u - \frac{\partial}{\partial x}u \right) = 0 \]
in the domain $\Omega \Def [0, 1]$ with Dirichlet boundary conditions
$\traceu(0) = 0$ and $\traceu(1) = 1$.
This problem can be seen as a 1D formulation of \Cref{pb:flux-mixed} when $\diffcoeff = 1$ and $f=0$.

In \Cref{fig:uniform-multiple-refinements}, we applied the \HDG[0] method for different values of \driftscalar and for a different mesh size $h$ of the uniform triangulation $\triangulation_h$.
We define the error functions
\[
  e_u \Def \left|u - u^h\right| \qquad \qquad
  e_{\traceu}\Def \left| \traceu -\traceu[h] \right|
\]

Because of \Cref{thm:unique-tau}, we know that there exists one and only one value for which $e_{\traceu}$ is identically zero. This can be seen also from a numerical point of view where we
identify one specific lower peak in the $\Lspace[\infty]{}$ error of the trace.
It is interesting to note that the peak correspond to a value of $\tau$ that decrease for smaller values of \driftscalar or $h$ (and, therefore, for smaller values of $\peclet$), in accordance with the values of \Cref{table:tau-val}.

\begin{figure}[p]
  \pgfplotstableread[col sep=comma]{data/uniform_plot_data_deg_0.csv}\uniformdatadegzero
  \centering
  \tikzsetnextfilename{error-vs-refinements-deg0-uniform}
  \begin{tikzpicture}
    \begin{groupplot}[group style={
                      group name=uniformerrorsdeg0,
                      group size= 3 by 3}, height=3.9cm, width=.30\linewidth,
                      xmode=log, ymode=log, xticklabel={\empty}, yticklabel={\empty},
                      ymin=0.0000000000001, ymax=100, xtick pos=left, ytick pos=left,
                      ytick={0.0000000000001, 0.0000000001, 0.0000001, 0.0001, 0.1, 100}]
        \nextgroupplot[title={$\driftscalar = 1$}, ylabel={$h = \nicefrac{1}{64}$},
		       yticklabel={\axisdefaultticklabellog}]
                \addplot [blue, thick] table [x={tau},
			  y={d0r6b1l2}]{\uniformdatadegzero};\label{plots:l2-unirefs}
		\addplot [red, thick, dashed] table [x={tau},
			  y={d0r6b1linftyhat}]{\uniformdatadegzero};\label{plots:linftyhat-unirefs}
		\coordinate (bot) at (rel axis cs:1,0);
        \nextgroupplot[title={$\driftscalar = 10$}]
                \addplot [blue, thick] table [x={tau},
			  y={d0r6b10l2}]{\uniformdatadegzero};
		\addplot [red, thick, dashed] table [x={tau},
			  y={d0r6b10linftyhat}]{\uniformdatadegzero};
        \nextgroupplot[title={$\driftscalar = 100$}]
                \addplot [blue, thick] table [x={tau},
			  y={d0r6b100l2}]{\uniformdatadegzero};
		\addplot [red, thick, dashed] table [x={tau},
			  y={d0r6b100linftyhat}]{\uniformdatadegzero};
        \nextgroupplot[ylabel={$h = \nicefrac{1}{256}$}, yticklabel={\axisdefaultticklabellog}]
                \addplot [blue, thick] table [x={tau},
			  y={d0r8b1l2}]{\uniformdatadegzero};
		\addplot [red, thick, dashed] table [x={tau},
			  y={d0r8b1linftyhat}]{\uniformdatadegzero};
        \nextgroupplot
                \addplot [blue, thick] table [x={tau},
			  y={d0r8b10l2}]{\uniformdatadegzero};
		\addplot [red, thick, dashed] table [x={tau},
			  y={d0r8b10linftyhat}]{\uniformdatadegzero};
        \nextgroupplot
                \addplot [blue, thick] table [x={tau},
			  y={d0r8b100l2}]{\uniformdatadegzero};
		\addplot [red, thick, dashed] table [x={tau},
			  y={d0r8b100linftyhat}]{\uniformdatadegzero};
        \nextgroupplot[ylabel={$h = \nicefrac{1}{1024}$}, xticklabel={\axisdefaultticklabellog},
                       yticklabel={\axisdefaultticklabellog}, xlabel={$\tau$}]
                \addplot [blue, thick] table [x={tau},
			  y={d0r10b1l2}]{\uniformdatadegzero};
		\addplot [red, thick, dashed] table [x={tau},
			  y={d0r10b1linftyhat}]{\uniformdatadegzero};
        \nextgroupplot[xticklabel={\axisdefaultticklabellog}, xlabel={$\tau$}]
                \addplot [blue, thick] table [x={tau},
			  y={d0r10b10l2}]{\uniformdatadegzero};
		\addplot [red, thick, dashed] table [x={tau},
			  y={d0r10b10linftyhat}]{\uniformdatadegzero};
        \nextgroupplot[xticklabel={\axisdefaultticklabellog}, xlabel={$\tau$}]
                \addplot [blue, thick] table [x={tau},
			  y={d0r10b100l2}]{\uniformdatadegzero};
		\addplot [red, thick, dashed] table [x={tau},
			  y={d0r10b100linftyhat}]{\uniformdatadegzero};
    \end{groupplot}

\path (uniformerrorsdeg0 c1r3.south west|-current bounding box.south)--
      coordinate(legendpos)
      (uniformerrorsdeg0 c3r3.south east|-current bounding box.south);
\matrix[
    matrix of nodes,
    anchor=north,
    draw,
    inner sep=0.2em,
    draw
  ]at([yshift=-0.2ex]legendpos)
  {
    \ref*{plots:l2-unirefs}& $\left\|e_u\right\|_{\Lspace{\triangulation}}$&[5pt]
    \ref*{plots:linftyhat-unirefs}& $\left\|e_{\traceu}\right\|_{\Lspace[\infty]{\faces}}$\\};
  \end{tikzpicture}
  \captionsetup{type=figure}
  \captionof{figure}{
      Errors obtained solving $\diver*{-u + \driftcoeff \nabla u} = 0$ on domain
      $\Omega = \left(0, 1\right)$ with $\traceu(0)=0$ and $\traceu(1)=1$ with \HDG[0].
  }
  \label{fig:uniform-multiple-refinements}
\end{figure}
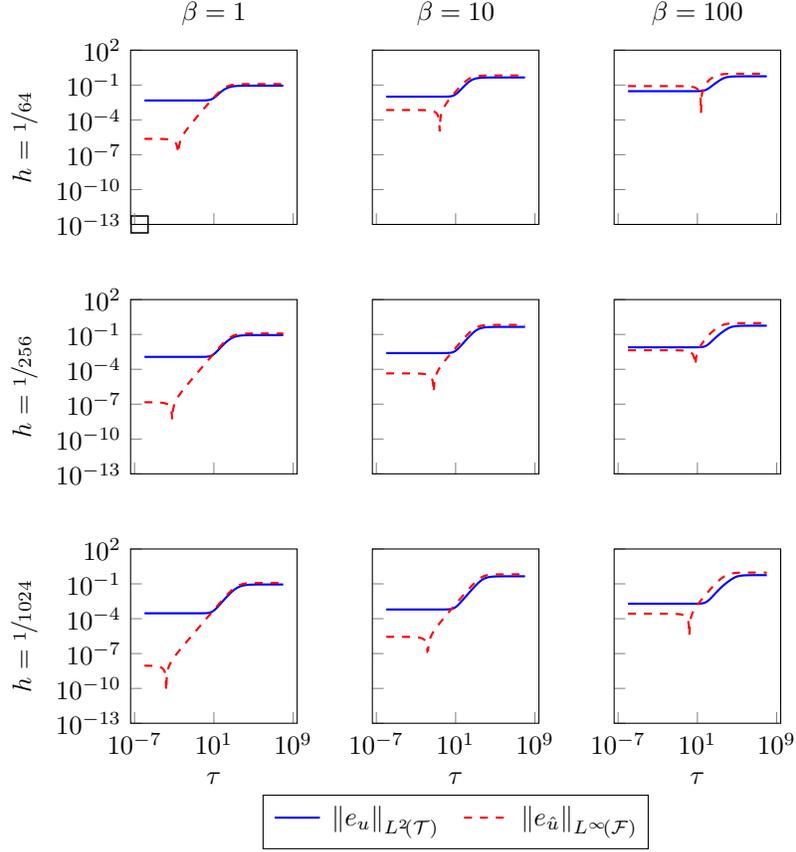
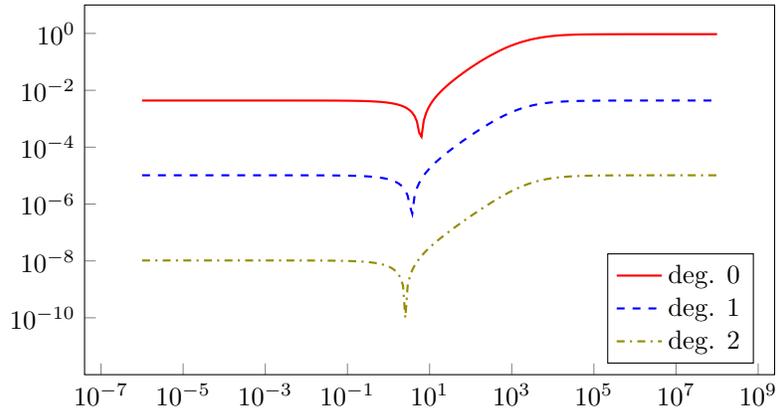
\begin{figure}[p]
  \pgfplotstableread[col sep=comma]{data/uniform_plot_data.csv}\uniformdata
  \centering
  \tikzsetnextfilename{error-vs-degrees-uniform}
  \begin{tikzpicture}
    \begin{axis}[
        height=6.5cm,
        width=.85\linewidth,
        xmode=log,
        ymode=log,
        ymin=1e-12,
        ymax=10,
        xtick pos=left,
        ytick pos=left,
        ytick={1e-10, 1e-8, 1e-6, 1e-4, 1e-2, 1},
        legend pos=south east
      ]
      \addplot [red, thick] table [x={tau}, y={d0r8b100linftyhat}]{\uniformdata};
      \addlegendentry{deg. 0}
      \addplot [blue, thick, dashed] table [x={tau}, y={d1r8b100linftyhat}]{\uniformdata};
      \addlegendentry{deg. 1}
      \addplot [olive, thick, dash dot] table [x={tau}, y={d2r8b100linftyhat}]{\uniformdata};
      \addlegendentry{deg. 2}
    \end{axis}
  \end{tikzpicture}
  \captionsetup{type=figure}
  \captionof{figure}{
      The $\left\|e_{\traceu}\right\|_{\Lspace[\infty]{\faces}}$ error obtained
      solving $\diver*{-u + \driftcoeff \nabla u} = 0$ on domain
      $\Omega = \left(0, 1\right)$, using elements of different degrees,
      with $\traceu(0)=0$ and $\traceu(1)=1$ for $h = \frac{1}{256}$ and
      $\driftscalar = 100$.
  }
  \label{fig:uniform-multiple-degree}
\end{figure}

For what concerns the error $e_u$, the plots in \autoref{fig:uniform-multiple-refinements} seems to
suggest to take the smaller possible value for $\tau$, but this conflicts with the fact that
the condensed linear system becomes less conditioned as soon as $\tau$ becomes close to zero (where
we get a singular matrix). In this prospective, $\tau_0$ is a value where the error $\left\|
e_u \right\|_{\Lspace{\triangulation}}$ is still reasonably small but the system is well
conditioned.

In \Cref{fig:uniform-multiple-degree}, instead, we shoe the behavior of $e_{\traceu}$ for the $\HDG[k]$
scheme using several different values of $k$. There we see that all three plots have a
similar behavior with a single minimum (corresponding to the value of $\tau_k$). As we have already seen
in \Cref{fig:delta-k}, the position of the minimum decrease increasing the degrees.

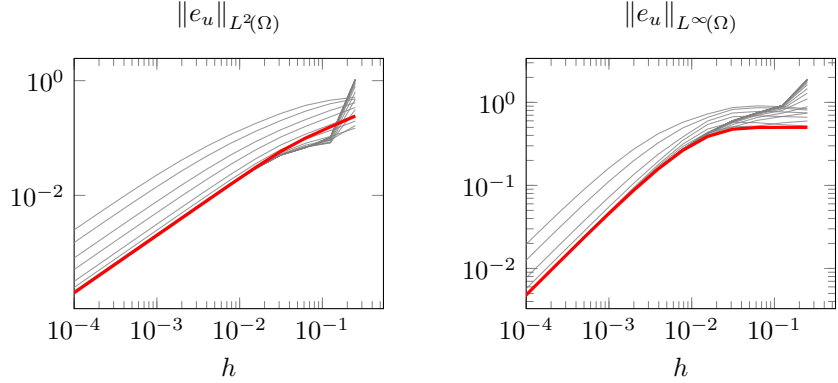
\begin{figure}[hpt]
  \pgfplotstableread[col sep=comma]{data/convergence_deg0_beta100.csv}\convdegzero
  \centering
  \makeatletter
  \tikzsetnextfilename{convergence-deg0-uniform}
  \begin{tikzpicture}
    \begin{groupplot}[group style={
                      group name=convplotdeg0,
                      group size= 2 by 2, vertical sep = 80pt, horizontal sep = .15\textwidth},
                      width=.45\textwidth, xmode=log, ymode=log, xmin=1e-4]
      \nextgroupplot[title={$\left\|e_u \right\|_{\Lspace{\Omega}}$}, xlabel={$h$}]
        \@for\tauval:={%
              0.01,0.018,0.032,0.056,0.1,0.18,0.32,0.56,1,1.8,3.2,5.6,10,18,32,%
              56,1e+02,1.8e+02,3.2e+02,5.6e+02,1e+03}%
          \do {%
	      \addplot [gray, opacity=0.8] table %
	          [x expr=1 / (2 ^ \thisrow{refinements}), y={tau\tauval_l2}]%
	          {\convdegzero};
	}
        \addplot [red, style={very thick}] table %
	          [x expr=1 / (2 ^ \thisrow{refinements}), y={tau0_l2}]%
	          {\convdegzero};
      \nextgroupplot[title={$\left\|e_u \right\|_{\Lspace[\infty]{\Omega}}$}, xlabel={$h$}]
        \@for\tauval:={%
              0.01,0.018,0.032,0.056,0.1,0.18,0.32,0.56,1,1.8,3.2,5.6,10,18,32,%
              56,1e+02,1.8e+02,3.2e+02,5.6e+02,1e+03}%
          \do {%
	      \addplot [gray, opacity=0.8] table %
	          [x expr=1 / (2 ^ \thisrow{refinements}), y={tau\tauval_linfty}]%
	          {\convdegzero};
	}
        \addplot [red, style={very thick}] table %
	          [x expr=1 / (2 ^ \thisrow{refinements}), y={tau0_linfty}]%
	          {\convdegzero};
    \end{groupplot}
  \end{tikzpicture}
\makeatother
\captionsetup{type=figure}
  \captionof{figure}{
    The error of the \HDG[0] method as a function of the cell size $h$ for a fixed value of
    the parameters $\alpha$ and $\beta$ ($\alpha= 1$, $\beta = 100$) and for $f = 0$. The red
    line is the result using $\tau = \tau_0$.
  }
\label{fig:convergence-deg0}
\end{figure}

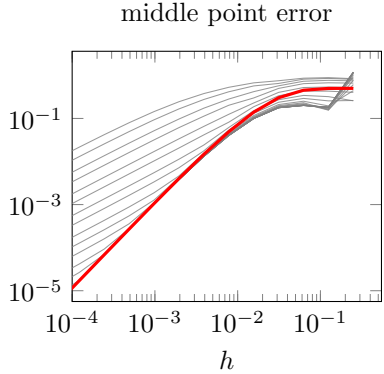
\begin{figure}[hpb]
  \pgfplotstableread[col sep=comma]{data/convergence_deg0_beta100.csv}\convdegzero
  \centering
  \makeatletter
  \tikzsetnextfilename{convergence-deg0-uniform-middlepoint}
  \begin{tikzpicture}
    \begin{axis}[%
        width=.45\textwidth,
        xmode=log,
        ymode=log,
        xmin=1e-4,
        title={middle point error},
        xlabel={$h$}
      ]
        \@for\tauval:={%
              0.01,0.018,0.032,0.056,0.1,0.18,0.32,0.56,1,1.8,3.2,5.6,10,18,32,%
              56,1e+02,1.8e+02,3.2e+02,5.6e+02,1e+03}%
          \do {%
	      \addplot [gray, opacity=0.8] table %
	          [x expr=1 / (2 ^ \thisrow{refinements}), y={tau\tauval_middle}]%
	          {\convdegzero};
	}
        \addplot [red, style={very thick}] table %
	          [x expr=1 / (2 ^ \thisrow{refinements}), y={tau0_middle}]%
	          {\convdegzero};
    \end{axis}
  \end{tikzpicture}
  \makeatother\captionsetup{type=figure}
  \captionof{figure}{
    The error of the \HDG[0] method for the same problem described in \Cref{fig:convergence-deg0}.
    In this case, we show the error on the middle point of every cell. The red
    line is the result using $\tau = \tau_0$.
  }
  \label{fig:convergence-deg0-middlepoint}
\end{figure}

\begin{figure}[hp]
  \pgfplotstableread[col sep=comma]{data/convergence_deg1_beta100.csv}\convdegone
  \centering
  \makeatletter
  \tikzsetnextfilename{convergence-deg1-uniform}
  \begin{tikzpicture}
    \begin{groupplot}[group style={
                      group name=convplotdeg0,
                      group size= 2 by 1, vertical sep = 80pt, horizontal sep = .15\textwidth},
                      width=.45\textwidth, xmode=log, ymode=log, xmin=1e-3]
      \nextgroupplot[title={$\left\|e_u \right\|_{\Lspace{\Omega}}$}, xlabel={$h$}]
        \@for\tauval:={%
              0.01,0.018,0.032,0.056,0.1,0.18,0.32,0.56,1,1.8,3.2,5.6,10,18,32,%
              56,1e+02,1.8e+02,3.2e+02,5.6e+02,1e+03}%
          \do {%
	      \addplot [gray, opacity=0.8] table %
	          [x expr=1 / (2 ^ \thisrow{refinements}), y={tau\tauval_l2}]%
	          {\convdegone};
	}
        \addplot [red, style={very thick}] table %
	          [x expr=1 / (2 ^ \thisrow{refinements}), y={tau_magic_l2}]%
	          {\convdegone};
      \nextgroupplot[title={$\left\|e_u \right\|_{\Lspace[\infty]{\Omega}}$}, xlabel={$h$}]
        \@for\tauval:={%
              0.01,0.018,0.032,0.056,0.1,0.18,0.32,0.56,1,1.8,3.2,5.6,10,18,32,%
              56,1e+02,1.8e+02,3.2e+02,5.6e+02,1e+03}%
          \do {%
	      \addplot [gray, opacity=0.8] table %
	          [x expr=1 / (2 ^ \thisrow{refinements}), y={tau\tauval_linfty}]%
	          {\convdegone};
	}
        \addplot [red, style={very thick}] table %
	          [x expr=1 / (2 ^ \thisrow{refinements}), y={tau_magic_linfty}]%
	          {\convdegone};
    \end{groupplot}
  \end{tikzpicture}
\makeatother
\captionsetup{type=figure}
  \captionof{figure}{
    The error of the \HDG[1] method as a function of the cell size $h$ for a fixed value of
    the parameters $\alpha$ and $\beta$ ($\alpha= 1$, $\beta = 100$) and for $f = 0$. The red
    line is the result using $\tau = \tau_1 = \frac{\diffcoeff}{h}\dtau_1$.
  }
\label{fig:convergence-deg1}
\end{figure}

\begin{figure}[hp]
  \pgfplotstableread[col sep=comma]{data/convergence_deg2_beta100.csv}\convdegtwo
  \centering
  \makeatletter
  \tikzsetnextfilename{convergence-deg2-uniform}
  \begin{tikzpicture}
    \begin{groupplot}[group style={
                      group name=convplotdeg0,
                      group size= 2 by 1, vertical sep = 80pt, horizontal sep = .15\textwidth},
                      width=.45\textwidth, xmode=log, ymode=log, xmin=1e-2]
      \nextgroupplot[title={$\left\|e_u \right\|_{\Lspace{\Omega}}$}, xlabel={$h$}]
        \@for\tauval:={%
              0.01,0.018,0.032,0.056,0.1,0.18,0.32,0.56,1,1.8,3.2,5.6,10,18,32,%
              56,1e+02,1.8e+02,3.2e+02,5.6e+02,1e+03}%
          \do {%
	      \addplot [gray, opacity=0.8] table %
	          [x expr=1 / (2 ^ \thisrow{refinements}), y={tau\tauval_l2}]%
	          {\convdegtwo};
	}
        \addplot [red, style={very thick}] table %
	          [x expr=1 / (2 ^ \thisrow{refinements}), y={tau_magic_l2}]%
	          {\convdegtwo};
      \nextgroupplot[title={$\left\|e_u \right\|_{\Lspace[\infty]{\Omega}}$}, xlabel={$h$}]
        \@for\tauval:={%
              0.01,0.018,0.032,0.056,0.1,0.18,0.32,0.56,1,1.8,3.2,5.6,10,18,32,%
              56,1e+02,1.8e+02,3.2e+02,5.6e+02,1e+03}%
          \do {%
	      \addplot [gray, opacity=0.8] table %
	          [x expr=1 / (2 ^ \thisrow{refinements}), y={tau\tauval_linfty}]%
	          {\convdegtwo};
	}
        \addplot [red, style={very thick}] table %
	          [x expr=1 / (2 ^ \thisrow{refinements}), y={tau_magic_linfty}]%
	          {\convdegtwo};
    \end{groupplot}
  \end{tikzpicture}
\makeatother
\captionsetup{type=figure}
  \captionof{figure}{
    The error of the \HDG[2] method as a function of the cell size $h$ for a fixed value of
    the parameters $\alpha$ and $\beta$ ($\alpha= 1$, $\beta = 100$) and for $f = 0$. The red
    line is the result using $\tau = \tau_2 = \frac{\diffcoeff}{h}\dtau_2$.
  }
\label{fig:convergence-deg2}
\end{figure}

In \Cref{fig:convergence-deg0},  \Cref{fig:convergence-deg1}, and \Cref{fig:convergence-deg2},
we compare the convergence of the $\HDG[k]$ method with $\tau=\tau_k$ against a fixed choice of $\tau$.
Indeed, we report as gray lines the convergence plots for several different fixed choice of $\tau$,
going from $10^{-2}$ up to $10^3$. In red, instead, we have the error of the \HDG method that
uses our proposed choice of $\tau$. The plots show that we have the same order of convergence in both the
$\Lspace{}$ space for the solution on the trace and in the $\Lspace[\infty]{}$ space for the trace.

Let us define for every cell $\cell \in \triangulation$ the point $x_\cell$ as the center of the
cell. We define the middle point error of $u^h$ as
\begin{equation}
 \max_{\cell \in \triangulation} \left|u\left(x_{\cell}\right) - u^h\left(x_{\cell}\right) \right| \eqdot
\end{equation}
\Cref{fig:convergence-deg0} shows the middle point error of the \HDG[0] method
as a function of $h$. In this case, we see that using $\tau = \tau_0$ the error decreases with one order of convergence more than with any other fixed choice of $\tau$. This is a property that is weel known for the Scharfetter--Gummel scheme applied to the finite volume methods.

\bibliographystyle{abbrv}
\bibliography{references}

\end{document}